\documentclass[11pt]{amsart}
\usepackage{amsmath,amssymb, amsthm, mathabx,amscd, amsfonts, amssymb, hyperref, mathrsfs, textcomp, epsfig, graphicx, IEEEtrantools, tikz, verbatim, subcaption, a4wide}
\newtheorem*{thm}{Theorem}

\newtheorem*{prop}{Proposition}
\newtheorem*{lemma}{Lemma}

\theoremstyle{definition}

\theoremstyle{remark}

\newcommand{\spin}{\mathfrak{s}}
\newcommand{\D}{\mathcal{D}}
\newcommand{\matZ}{\mathbb{Z}}
\newcommand{\matQ}{\mathbb{Q}}

\begin{document}

\title[Adjunction and Davis]{Adjunction inequalities \\ and the Davis hyperbolic four-manifold}

\author{Francesco Lin}
\address{Department of Mathematics, Columbia University} 
\email{flin@math.columbia.edu}

\author{Bruno Martelli}
\address{Dipartimento di Matematica, Universit\`a di Pisa} 
\email{bruno.martelli@unipi.it}
\begin{abstract}
The Davis hyperbolic four-manifold $\D$ is not almost-complex, so that its Seiberg-Witten invariants corresponding to zero-dimensional moduli spaces are vanishing by definition. In this paper, we show that all the Seiberg-Witten invariants involving higher-dimensional moduli spaces also vanish. Our proof involves the adjunction inequalities corresponding to 864 genus two totally geodesic surfaces embedded inside $\D$.
\end{abstract}

\maketitle

\section*{Introduction}

LeBrun conjectured in \cite{LeBrun} that the Seiberg-Witten invariants \cite{Mor} of hyperbolic four-manifolds vanish. This question is closely related to several open problems in the field, and most notably it implies that there do not exist hyperbolic four-manifolds which are surface bundles over surfaces (cf. \cite{Reid}, and see also the recent construction of atoroidal surface bundles over surfaces \cite{KL}).
\par
The Seiberg-Witten invariants of hyperbolic four-manifolds are especially difficult to determine because the description of the latter is usually not well-suited for the standard tools in the subject. Examples of manifolds with vanishing Seiberg-Witten invariants were provided in \cite{AgolLin, BFS} by exhibiting separating $L$-spaces; of course, this Floer-theoretic approach is in general very hard to implement for a given manifold of interest.
\\
\par
In this paper we consider the Davis hyperbolic four-manifold \cite{Davis}, which we denote by $\D$. It is obtained by identifying the opposite faces of a $120$-cell (a remarkable four-dimensional regular polytope), and is arguably the closed hyperbolic four-manifold with the simplest description. Its topological and geometric properties have been studied in detail in \cite{RT}; in particular it is shown that
\begin{equation}\label{homology}
H_1(\D;\mathbb{Z})=\mathbb{Z}^{24}\text{ and }H_2(\D;\mathbb{Z})=\mathbb{Z}^{72}.
\end{equation}
All hyperbolic four-manifolds have signature zero by the Hirzebruch signature formula, so in particular $b^+(\D)=36\geq 2$ and the Seiberg-Witten invariants are defined. Our main result is then the following.
\begin{thm}
The Seiberg-Witten invariants of the Davis hyperbolic four-manifold $\D$ are all zero. Here we consider all the Seiberg-Witten invariants obtained by evaluating classes in $\Lambda^*(H_1/\mathrm{tors})\otimes\mathbb{Z}[U]$ (the cohomology ring of the moduli space of irreducible configurations) on moduli spaces of solutions, for all spin$^c$ structures. 
\end{thm}
Let us point out that in \cite{RT} it is shown that $\D$ is not almost-complex, so that there are no spin$^c$ structures with zero-dimensional moduli spaces and therefore the corresponding Seiberg-Witten invariants are zero by definition. Also, the statement holds for either orientation because $\D$ admits an orientation reversing isometry.
\\
\par
Unlike the three-manifold based approach of \cite{AgolLin,BFS}, our proof will exploit the topology of the surfaces contained in $\D$. In \cite{RT} the authors describe a collection of 864 embedded genus two totally geodesic surfaces contained in $\D$; by Chern-Weil theory, such surfaces have self-intersection zero because the Levi-Civita connection on the normal bundle is flat. Each such surface $S$ therefore implies the adjunction inequality
\begin{equation}\label{adjunction}
|\langle c_1(\spin),[S]\rangle|\leq 2g(S)-2=2
\end{equation}
for any spin$^c$ structure $\spin$ admitting some non-vanishing Seiberg-Witten invariant. This follows from the same proof as in the zero-dimensional case \cite[Ch. 40]{KM}, as the argument shows that the moduli space is empty for a suitable metric if such inequality does not hold. 
\par
On the other hand, the dimension of the moduli space of solutions for $\spin$ is given \cite[Ch. 1.4]{KM} by
\begin{equation*}
\frac{c_1(\spin)^2-2\chi-3\sigma}{4}=\frac{c_1(\spin)^2-52}{4}
\end{equation*}
as $\chi(\D)=26$ by (\ref{homology}) and $\sigma(\D)=0$. The Theorem is then proved by showing that for the spin$^c$ structures satisfying all the adjunction inequalities (\ref{adjunction}), one has
\begin{equation}\label{ineq}
c_1(\spin)^2 < 52,
\end{equation}
so that the corresponding moduli space has negative expected dimension, and therefore the corresponding Seiberg-Witten invariants all vanish. In fact, we will prove that $|c_1(\spin)^2| \leq 32$.
\\
\par
The genus two surfaces we consider have a very nice description in terms of the geometry of the 120-cell; we review it in Section \ref{genustwo}. Given this, the main technical complication is the sheer complexity of the problem under consideration arising from the size of $H_2$. To give a rough idea, the simplest $72$ inequalities we consider will allow us reduce ourselves to study spin$^c$ structures parametrized by the standard cube in $\mathbb{Z}^{72}$ with coordinates in $\{-1,0,1\}$. This contains $3^{72}\approx 2\times 10^{34}$ lattice points; enumerating and testing all the spin$^c$ structures corresponding to these points is out of reach for any computer available today. However, we will exploit the rich combinatorial structure of the collection of surfaces and the corresponding additional constraints together with the large symmetry group of the Davis hyperbolic manifold to reduce the size of the set from $\approx 2\times 10^{34}$ to a much smaller number, so that the whole computation can be completed in roughly one month of CPU time; this is explained in Section \ref{computation}.
\\
\par
Because a hyperbolic four-manifold $X$ always has $\chi(X)>0$ (by Chern-Gauss-Bonnet) and $\sigma(X)=0$, the overall strategy of our proof can be in principle applied to more general situations of hyperbolic four-manifolds containing large collections of embedded totally geodesic surfaces. For context, let us point out that all compact arithmetic hyperbolic $4$-manifolds are of the simplest type \cite[Proposition 6.4.5]{Morris} and therefore admit infinitely many \textit{immersed} totally geodesic surfaces.
\par
Of course, the computational complexity of the argument in the (arguably) simplest case of the Davis manifold $\D$ suggests that it might be challenging to implement the same strategy in other examples. However, in more general situations it could be helpful to consider also the higher type adjunction inequalities proved by Ozsv\'ath and Szab\'o in \cite{OzSz}, which provide some refinements of (\ref{adjunction}) in terms of the map induced by the inclusion $S\hookrightarrow X$ in $H_1$.
\\
\par
\textit{Acknowledgements.} The first author was partially supported by NSF grant DMS-2203498. The second author was
partially supported by INdAM through GNSAGA, and by MUR through the PRIN project ``Geometry and topology of manifolds''.
\vspace{0.4cm}

\section{Genus two surfaces in the Davis hyperbolic four-manifold}\label{genustwo}
In this section we describe the Davis manifold $\D$, its isometry group, and many embedded totally geodesic surfaces of genus two. The isometry group and the surfaces are taken from \cite{RT}. Here we equip these surfaces with an orientation and we compute their algebraic intersections.

\subsection{The Davis manifold}
The $120$-cell $P$ is a regular four-dimensional polytope having $120$ regular dodecahedra as facets. We refer the reader to \cite{SchSeg} for a nice discussion of its beautiful properties. In particular, its projection to $S^3$ is the Voronoi cellulation of the binary icosahedral group $I_{120}^*$ (the preimage under the double cover $S^3=SU(2)\rightarrow SO(3)$ of the subgroup $A_5$ corresponding to the orientation preserving isometries of the dodecahedron). It has $720$ pentagonal ridges, $1200$ edges and $600$ vertices.
\par
The $120$ facets come in $60$ pairs of parallel facets, and the Davis hyperbolic four-manifold $\D$ is obtained by identifying these pairs by translation. Under this identification, the $720$ ridges are identified in $144$ groups of $5$, the $1200$ edges in $60$ groups of $20$, and all the vertices are identified into a single vertex $v$. Therefore
\begin{equation*}
\chi(\D) = 1-60+144-60+1 = 26.
\end{equation*}
The hyperbolic metric is obtained by realizing the $120$-cell as a polytope in hyperbolic space with dihedral angles of $2\pi/5$.
\par

\subsection{The isometry group}
The orientation-preserving isometry group of $\D$ fits in the short exact sequence
\begin{equation*}
1\rightarrow \mathrm{Isom}^+(P)\rightarrow \mathrm{Isom}^+(\D)\rightarrow \mathbb{Z}/2\rightarrow 1
\end{equation*}
where $|\mathrm{Isom}^+(P)|=14400$. This sequence splits, but not canonically; the elements in $\mathrm{Isom}^+(\D)\setminus \mathrm{Isom}^+(P)$ are called \emph{inside-out isometries} and they exchange the center of $P$ with the unique vertex $v$ of the cell decomposition \cite{RT}.
\\
\par
We now describe some totally geodesic genus two surfaces embedded in $\D$, following \cite{RT}. We refer the reader to \cite{RT,SchSeg} for pictures. There are in fact four kinds of such surfaces. 

\subsection{Two classes of surfaces}
The first kind of surfaces is defined as follows. Pick a dodecahedral facet $D$ of $P$ and two opposite pentagonal faces of it. Consider the geodesic segment in $D$ connecting orthogonally these two faces.
The segment can be prolonged in the adjacent dodecahedral facets of $P$ sharing the pentagonal faces; each of these paths closes up after traversing $10$ dodecahedra. In the spherical picture, this corresponds to a great circle passing through $10$ of the $120$ points in $I_{120}^*$; it is shown in \cite{SchSeg} that these circles can be used to define discrete versions of the Hopf fibration $S^3\rightarrow S^2$. In the hyperbolic version of $P$, this path is the boundary of a regular hyperbolic decagon with angles $2\pi/5$ contained in $P$, which under the identification gives rise to an embedded totally geodesic surfaces of genus $2$ in $\D$.

There are $120\times 6/10=72$ such closed paths (where $6$ is the number of pairs of opposite pentagonal faces in a dodecahedron), and they produce 72 embedded totally geodesic surfaces $a_1, \ldots, a_{72}$ of genus $2$ in $\D$. 
\\
\par
We now define the second kind of surfaces. We first note that every ridge in $P$ is a regular pentagon with angle $\pi/5$. Each of the closed paths just considered crosses $10$ pentagonal ridges at their center; it turns out that these are identified in $2$ groups of $5$, which alternate along the closed path, and give rise to two pentagons in $\D$ whose union is an embedded totally geodesic surface of genus 2. The 72 closed paths give rise in this way to 72 embedded totally geodesic surfaces $A_1, \ldots, A_{72}$ in $\D$, again of genus two.
\\
\par
Summing up, there are 72 great circles, and each gives rise to two surfaces $a_i$ and $A_i$, for $i=1,\ldots, 72$. The surfaces $a_i$ and $A_j$ intersect only if $i=j$, and they do so transversely into two points, that are the centers of two pentagons, with the same sign. Therefore their algebraic intersection is $\pm 2$, depending on the orientations that we assign to $a_i$ and $A_i$. Assigning an orientation to $a_i$ is equivalent to assigning an orientation to the corresponding great circle. There does not seem to be a canonical way to orient all the 72 great circles, so we have chosen arbitrarily an orientation for them, and we have oriented the surfaces $A_i$ accordingly so that
\begin{equation}
Q( a_i, A_j)=2 \delta_{ij}.
\end{equation}
Here $Q$ denotes the intersection form on $H_2(\D, \matZ)$, and for simplicity we will drop the square brackets from the notation of the homology class of surfaces here and in what follows.
\par
We know from \cite{RT} that the oriented surfaces $A_i$ forms a basis for $H_2(\D, \matQ)$, and they generate a subgroup of $H_2(\D, \matZ) \cong \matZ^{72}$ of index $2^{36}$. Any inside-out isometry of $\D$ interchanges the sets $\{a_i\}$ and $\{A_i\}$, however without preserving the index $i$. Therefore the cycles $\{a_i\}$ also generate a subgroup of $H_2(\D, \matZ)$ of index $2^{36}$, coherently with the fact that the $72 \times 72$ diagonal matrix $[Q(a_i, A_j)]$ has determinant $2^{72}$. Furthermore, we have
\begin{equation}\label{intform}
Q( a_i, a_j)=
\begin{cases}
0 \text{ if the corresponding great circles intersect,}\\
\pm 1 \text{ otherwise. }
\end{cases}
\end{equation}
To see this, we note that if the great circles intersect the two surfaces $a_i$ and $a_j$ can be isotoped to be disjoint, while if they do not intersect the two surfaces $a_i$ and $a_j$ intersect transversely in the center of the 120-cell $P$. The sign of the intersection is the linking number of the two oriented great circles.
For a given $a_i$, there are exactly $26$ classes $a_j$ that have intersection $0$ with it: the class $a_i$ itself, and those corresponding to the $10\times 5/2=25$ great circles intersecting it. 

Summing up, the $72\times 72$ symmetric intersection matrix $Q_{ij} = Q(a_i,a_j)$ contains only the numbers $0,1,-1$, with exactly $25$ zeroes in every line and column, and it has determinant $2^{72}$. We have determined $Q$ explicitly using a computer code.

\subsection{Two more classes of surfaces}
In order to find more constraints, we will need more surfaces. We keep following \cite{RT}. 

The third class of surfaces is constructed as follows: consider for each dodecahedral facet $D$ of $P$ two opposite pentagonal ridges. There are exactly $10$ edges of $D$ that do not intersect these two opposite faces; these form an `equator'. The identifications of $P$ are such that the edges in $D$ that are identified are precisely the opposite ones (via translation). Because of this, by taking the cone over the center of $D$, each equator determines a piecewise smooth genus-two surface $b_i$ with a removable self-intersection at the cycle of vertices $v$. In particular, after a small perturbation $b_i$ becomes a smooth embedded surface of genus $2$ in $\D$. There are $120\times 6/2=360$ such cycles, where the factor of $2$ comes from the fact that opposite facets of $P$ are identified in $\D$. As it will be explained later, each $b_i$ is homologous to a totally geodesic genus two embedded surface. We assign to each $b_i$ an arbitrary orientation. From this concrete description the algebraic intersection with the surfaces $a_i$ is readily determined to be
\begin{equation}
Q( a_i, b_j)=
\begin{cases}
\pm 1 \text{ if the great circle corresponding to $a_i$ intersects $D$,}\\
0 \text{ otherwise; }
\end{cases}
\end{equation}
in particular, for each $b_j$ there are exactly $6$ classes $a_i$ with non-zero intersection with it. 
\\
\par
For the fourth and final class, we consider the $360$ classes $B_i$ obtained from the $b_i$ by any inside-out isometry. These classes can be described explicitly as follows. Consider a pentagonal ridge $F$ of $P$ (which is shared by exactly two facets). Each edge in $P$ belongs to three dodecahedral facets; therefore there are five facets of $P$ intersecting $F$ exactly in an edge (one for each edge $e$ of $F$); in each of these five facets $D$, consider the hexagon obtained by slicing through the plane passing through the edge $e$ and the center of $D$. We consider in this hexagon the diagonal connecting two opposite vertices which is parallel to $e$; notice that these two opposite vertices are midpoints
of opposite edges in $D$. The union of these diagonals in the five hexagons corresponding to the pentagonal ridge $F$ determines a `small circle' in $S^3$, the convex hull (in the hyperbolic $120$-cell) of which is a totally geodesic regular pentagon with angle $\pi/5$. By taking the union of the two pentagons corresponding to opposite facets in $P$ we obtain a totally geodesic surface $B_i$ of genus $2$. Each of these $B_i$ has two corresponding small circles, which are the two solid pentagons in \cite[Figure 3]{RT}. The images of these $720/2=360$ surfaces under an inside-out isometry are the totally geodesic representatives of the $b_i$ mentioned above. The intersections with the $a_i$ are determined explicitly to be
\begin{equation}
Q( a_i, B_j)=
\begin{cases}
\pm 2 \text{ if the great circle of $a_i$ intersects the ridges of $B_j$ in their center;}\\
\pm 1 \text{ if the great circle of $a_i$ intersects the small circles of $B_j$ in one point each; }\\
0\text{ otherwise. }
\end{cases}
\end{equation}
Here, in the second case the two intersections points are identified in $\D$ to a single transverse intersection point. From our description of the great circles we see that each class $a_i$ has intersection $\pm2$ with exactly $10/2=5$ classes $B_j$. Furthermore, we can see that each $a_i$ has intersection $\pm1$ with exactly $100$ classes $B_j$ as follows. Consider the $10$ dodecahedra intersected by the great circle of $a_i$.  There are $10\times 20$ edges not belonging to the pentagonal ridges in which these dodecahedra intersect, and each pair of such edges opposite in $P$ determines a $B_j$ intersecting $a_i$ in $\pm1$. Dually, each $B_j$ has intersection $\pm2$ with one $a_i$ and $\pm1$ with $20$ classes $a_i$. 
\par
The explicit form of the $72\times 360$ matrices $Q( a_i, b_j)$ and $Q( a_i, B_j)$ was determined using a computer code after suitably enumerating and orienting the surfaces $b_i$ and $B_i$.

\vspace{0.4cm}

\section{Computational aspects}\label{computation}

The main goal of this section is to describe the implementation of a computer-assisted proof of the following.

\begin{prop}\label{inequalities}
    If a spin$^c$ structure $\spin$ on $\D$ satisfies 
    \begin{equation}\label{adjunction}
|\langle c_1(\spin),[S]\rangle|\leq 2
\end{equation}
for each of the 864 surfaces $S$ constructed in Section \ref{genustwo}, then $|c_1(\spin)^2| \leq 32$.
\end{prop}

This implies the main theorem stated in the Introduction, as explained there.

\subsection{The linear inequalities}
Recall that we have 72 surfaces $a_i$, 72 surfaces $A_i$, 360 surfaces $b_i$, and 360 surfaces $B_i$.
We fix the basis $a_i$ of $H_2(\D, \matQ)$. The intersection form is represented by the matrix $Q_{ij} = Q(a_i,a_j)$. 

Let $\spin$ be a class that satisfies the hypothesis of Theorem \ref{inequalities}, and $x = PD(c_1(\spin))$. Our aim is to prove that 
$$|Q(x,x)| \leq 32.$$
We use the basis $a_i$ and write
$$x = \sum_i x_ia_i$$
for some $x_i \in \matQ$ (since the $a_i$ do not form a basis for $H_2(\D, \matZ)$, the coefficients $x_i$ are not necessarily integers). Recall that $\D$ is a spin manifold \cite{RT}, so the elements $c\in H^2(\D, \matZ)$ which are the first Chern class of a spin$^c$ structure are exactly the even classes (that is, those divisible by two); our analysis is significantly simplified by the following.

\begin{lemma}
We have $x_i\in\{-1,0,1\}$ for all $i$.
\end{lemma}
\begin{proof}
Since $c_1(\spin)$ is even, its dual $x$ also is, and hence
\begin{equation*}
Q( x,A_j)=Q( \sum x_i a_i,A_j) =2x_j
\end{equation*}
is an even number for all $j=1,\dots, 72$, so that all the $x_j$ are integers. The conclusion follows because $|Q( x,A_j)|\leq 2$ by (\ref{adjunction}).
\end{proof}

We can therefore restrict our attention to the $3^{72}$ classes $x$ whose coordinates lie in the cube $\{-1,0,1\}^{72}$. During the proof of the lemma we have exploited that $|Q(x,A_i)| \leq 2$. We still have to impose the $72+360+360$ inequalities
$$|Q(x,a_i)| \leq 2, \quad |Q(x,b_i)| \leq 2, \quad |Q(x,B_i)| \leq 2.
$$
The first type of inequality may be rewritten as
$$\big|\sum_j Q_{ij} x_i\big| \leq 2.$$
We discovered in Section \ref{genustwo} that there are precisely 47 non-zero coefficients among the $Q_{ij}$, for any fixed $i$, and these are all $\pm 1$. We also found that, for each $b_j$, there are exactly $6$ classes $a_i$ with non-zero intersection with it: therefore the second type of inequalities takes the simple form
\begin{equation} \label{bj}
|x_{i_1}\pm x_{i_2}\pm x_{i_3}\pm x_{i_4}\pm x_{i_5}\pm x_{i_6}|\leq2
\end{equation}
for six distinct indices $i_1,\dots,i_6$ that depend on $b_j$. Analogously, the inequality corresponding to a $B_j$ takes the form
\begin{equation} \label{Bj}
|2x_{i_0}\pm x_{i_1} \pm x_{i_2} \cdots \pm x_{i_{20}}|\leq2
\end{equation}
for some distinct indices $i_0,\dots,i_{20}$ depending on $B_j$.

\subsection{The implementation}
We have written a code in Sage, available from \cite{code}, which can be used to investigate all the vectors $x \in \{-1,0,1\}^{72}$ that satisfy the $72+360+360$ inequalities described above. Since $x$ represents an even class, and $Q$ is also itself even, we know that $Q(x,x)$ is divisible by eight. By running our code we discover that there are many classes $x$ that satisfy the given inequalities, and for all these classes the absolute value of the self-intersection $Q(x,x)$ is either $0$, $8$, $16$, $24$, or $32$. This proves the Proposition.

We say a few words on the code. The number of lattice points to be investigated is $3^{72}\approx 2\times 10^{34}$, and it is impossible to test the $792$ inequalities separately on each point in a reasonable time. We have not found in the literature any general method to reduce the complexity of such a problem; to considerably reduce the CPU time of the running code, we have proceeded instead as follows. First, we take advantage of the many isometries of $P$. The 14400 isometries of $\D$ act on the set $\{-1,0,1\}^{72}$, and it is also possible to invert all signs, so the problem has 28800 symmetries overall; we say that an element $x \in \{-1,0,1\}^{72}$ is \emph{minimal} if it is minimal (with respect to the lexicographic order) in its orbit. We also say that it is \emph{admissible} if it satisfies all the $792$ inequalities. Of course it is sufficient for us to test only minimal admissible elements $x$.

We have considered the variable $x_i \in \{-1,0,1\}$ iteratively, starting from $i=1,2,3, \ldots$ trying to cut as many dead branches as possible, to reduce the number of elements under investigation from $3^{72}$ to a much smaller number. Every time we assign a value to some $x_i$, we do a number of tests to see if the initial segment $(x_1,\ldots, x_i)$ has some obstruction that would prevent any completion $(x_1,\ldots, x_{72})$ to be both minimal and admissible. The minimality test is easy to implement, and it is very effective due to the large number of isometries; the admissibility tests are the more effective when the corresponding inequality has less non-zero coefficients: in particular the inequalities \eqref{bj} 
are very much effective, and the \eqref{Bj} are also crucial to reduce the running time to a resonable size. We really needed to use all the inequalities to complete the task in reasonable time. The count was completed roughly in a month of CPU time.

\vspace{0.5cm}

\bibliographystyle{alpha}
\bibliography{biblio}

\end{document}